\documentclass[twoside,reqno,10pt]{amsart}

 \usepackage{amsmath}
 \usepackage{url}

\newcommand{\fdiff}[4]{\ensuremath{ \upsilon^{#4 }_{ #3 }  #1  \left(  #2 \right)  }}
\newcommand{\fdiffplus}[3]{ \fdiff {#1}{#2}{+}{#3} }
\newcommand{\fdiffs}[2]{\ensuremath{ \upsilon^{#1 }_{ #2 }    }}
\newcommand{\fdiffmin}[3]{ \fdiff {#1}{#2}{-}{#3} }
\newcommand{\fdiffpm}[3]{ \fdiff {#1}{#2}{\pm}{#3} }
\newcommand{\fracvar}[4]{\ensuremath{ \upsilon_{ #3 }^{ #4} \left[ #1 \right] \left(  #2 \right)   }}
\newcommand{\fracvarplus}[3]{ \fracvar {#1}{#2}{#3}{\epsilon+} }
\newcommand{\fracvarmin}[3]{ \fracvar {#1}{#2}{#3}{\epsilon -} }
\newcommand{\fracvarpm}[3]{ \fracvar {#1}{#2}{#3}{\epsilon \pm} }

\newcommand{\llim}[3]{\ensuremath{ \lim\limits_{ #1 \rightarrow #2} #3 }}
\newcommand{\fclass}[2]{\ensuremath{  \mathbb{#1}^{\, #2} }}

\newcommand{\soc}[2]{\ensuremath{ \chi_{#1}^{#2} }}

\newcommand{\holder}[1]{\fclass{H}{#1} }

\newcommand{\osc}[4]{\ensuremath{ \mathrm{osc}_{ #3 }^{ #4} [ #1 ] \left(  #2 \right)   }}

\newcommand\smallO{
	\mathchoice
	{{\scriptstyle\mathcal{O}}}
	{{\scriptstyle\mathcal{O}}}
	{{\scriptscriptstyle\mathcal{O}}}
	{\scalebox{.7}{$\scriptscriptstyle\mathcal{O}$}}
}

\newcommand{\bigoh}[1]{ \ensuremath{ \smallO \left(  #1 \right) }   }

\newcommand{\deltaop}[4]{\ensuremath{ \Delta_{ #3 }^{ #4} \left[ #1 \right] \left(  #2 \right)   }}
\newcommand{\deltaplus}[2]{ \deltaop {#1}{#2}{\epsilon}{+} }
\newcommand{\deltamin}[2]{ \deltaop {#1}{#2}{\epsilon}{-} }

\newcommand{\sep}{,}
\newcommand{\epnt}{\; .}
\newcommand{\ecma}{\; ,}

\newcommand{\frdiffii}[3]{\ensuremath{\,_{#3}\mathbf{I}^{#1}_{#2}}}

\newcommand{\frdiffiix}[2]{\frdiffii{#1}{x}{#2}}

\newtheorem{theorem}{Theorem}
\newtheorem{lemma}{Lemma}
\newtheorem{corollary}{Corollary}
\newtheorem{proposition}{Proposition}

\newtheorem{condition}{Condition}

\newtheorem{definition}{Definition}
\newtheorem{remark}{Remark}
\newtheorem{example}{Example}

	\newcommand{\bigohone}[1]{ \ensuremath{ \smallO_{#1}  }   }
	
\newcommand{\svar}[4]{\ensuremath{ \mathcal{S}_{ #3 }^{ #4} \left[ #1 \right] \left(  #2 \right)   }}
\newcommand{\svarplus}[3]{ \svar{#1}{#2}{\epsilon+}{#3} }

\newcommand{\svarpm}[3]{ \svar{#1}{#2}{\epsilon \pm}{#3} }

	\newcommand{\holderspect}{ \ensuremath{   \mathbb{E}^{\pm}\, } }	
	 \newcommand{\svars}[2]{\ensuremath{ \mathcal{S}_{ #1 }^{ #2}    }}

 \usepackage{graphicx, amssymb}
 
 \usepackage{tikz}
 \usepackage{tikz-cd}
 \usetikzlibrary{matrix, calc, arrows, cd}
 
\begin{document}
  
   \title[Fractional velocity as a tool for non-linear problems] {Fractional velocity as a tool for the study of non-linear problems}
   
   \author {Dimiter Prodanov}
   \address{Correspondence: Environment, Health and Safety, IMEC vzw, Kapeldreef 75, 3001 Leuven, Belgium
   e-mail: Dimiter.Prodanov@imec.be, dimiterpp@gmail.com}

 
   	\begin{abstract}
   		
   	Singular functions and, in general, H\"older functions represent conceptual models of nonlinear physical phenomena.  
   	The purpose of this survey is to demonstrate the applicability of fractional velocity as a tool to characterize Holder and in particular singular functions.
   	Fractional velocities are defined as limit of  the difference quotient of a fractional power and they generalize the local notion of a derivative.
    On the other hand, their properties contrast some of the usual properties of derivatives. 
    One of the most peculiar properties of these operators is that the set of their non trivial values is disconnected.
    This can be used for example to model instantaneous interactions, for example  Langevin dynamics. 
   	Examples are given by the De Rham and Neidinger's functions, represented by iterative function systems.
	Finally the conditions for equivalence with the Kolwankar-Gangal local fractional derivative are investigated. 
  
   	 \medskip
   		
   		
 \smallskip
   		

 \end{abstract}

	   	\maketitle  
	   	
   	
   	\section{Introduction}
   	\label{seq:intro}

	Non-linear and fractal physical phenomena are abundant in nature \cite{Mandelbrot1982, Mandelbrot1989}.	  
 	Examples of non-linear phenomena can be given by the continuous time random walks resulting in fractional diffusion equations \cite{Metzler2004}, 
 	fractional conservation of mass  \cite{Wheatcraft2008} or non-linear viscoelasticity\cite{Caputo1971, Mainardi1997}.
 	Such models exhibit global dependence through the action of the nonlinear convolution operator (i.e differintegral).	
 	Since this setting opposes the principle of locality there can be problems with the interpretation of the obtained results. 
 	In most circumstances such models can be treated as asymptotic as it has been demonstrated for the time-fractional continuous time random walk \cite{Gorenflo2008}. 
 	The asymptotic character of these models leads to the realization that they  describe \textit{mesoscopic} behavior of the concerned systems.
 	The action of fractional differintegrals on analytic functions results in H\"older functions representable by fractional power series (see for example \cite{Oldham1974}).
   
  	Fractals are becoming essential components in the modeling and simulation of natural phenomena encompassing many temporal or spatial scales \cite{Schroeder1991}.
  	The irregularity and self-similarity under scale changes are the main attributes of the morphologic complexity of cells and tissues \cite{Losa1996}.
  	Fractal shapes are frequently built by iteration of function systems via recursion \cite{Hutchinson1981, Darst2009}.
  	On the other hand fractal shapes observable in natural systems  typically span only several recursion levels.
  	This fact draws attention to one particular class of  functions, called \textit{singular}, which are differentiable but for which at most points the derivative vanishes. 
   	There are fewer tools for the study of singular functions since one of the difficulties is that for them the Fundamental Theorem of calculus fails and hence they can not be represented by a non-trivial differential equation. 
   
   	Singular signals can be considered as toy-models for strongly-non linear phenomena. 
   	Mathematical descriptions of strongly non-linear phenomena necessitate relaxation of the assumption of differentiability \cite{Nottale2010}. 
   	While this can be achieved also by fractional differintegrals, or by multiscale approaches \cite{Cresson2016}, the present work focuses on local descriptions in terms of limits of difference quotients \cite{Cherbit1991} and  non-linear scale-space transformations \cite{Prodanov2016b}. 
   	The reason for this choice is that locality provides a direct way of physical interpretation of the obtained results.   		
   	In the old literature, difference quotients of functions of fractional order have been considered at first by du Bois-Reymond  \cite{BoisReymond1875} and Faber \cite{Faber1909} in their studies of the point-wise differentiability of functions. 
   	While these initial developments followed from purely mathematical interest, later works were inspired from physical research questions. Cherbit \cite{Cherbit1991} and later on Ben Adda and Cresson  \cite{Adda2001} introduced the notion of \textsl{fractional velocity} as the limit of the fractional difference quotient. 
   	Their main application was the study of fractal phenomena and physical processes for which the instantaneous velocity was not well defined \cite{Cherbit1991}.

 	This work will further demonstrate applications to singular functions.
 	Examples are given by the De Rham and Neidinger's functions, represented by iterative function systems.
 	In addition, the form of the Langevin equation is examined for the requirements of path continuity. 	
	Finally,  the relationship between fractional velocities and the localized versions of fractional derivatives in the sense of Kolwankar-Gangal will be demonstrated.


\section{General definitions and notations}
  	\label{sec:definitions}
  	  	
  	The term \textit{function}  denotes a mapping $ f: \mathbb{R} \mapsto \mathbb{R} $ (or in some cases $\mathbb{C} \mapsto \mathbb{C}$). 
  	The notation $f(x)$ is used to refer to the value of the mapping at the point \textit{x}.
    The term \textit{operator}  denotes the mapping from functional expressions to functional expressions.
  	Square brackets are used for the arguments of operators, while round brackets are used for the arguments of functions.	
    $Dom[f]$ denotes the domain of definition of the function $f(x)$.
    The term Cauchy sequence will be always interpreted  as a null sequence.
 
 	BVC[I] will mean that the function $f$ is continuous of bounded variation (BV) in the interval \textit{I}.
 	
  	 \begin{definition}[Asymptotic $\smallO$ notation]
   	 	\label{def:bigoh}
   	 	The notation  $\bigoh{x^\alpha}$ is interpreted as the convention that 
   	 	$$
   	 	\llim{x}{0}{ \frac{\bigoh{x^\alpha}}{x^\alpha} } =0 
   	 	$$
   	 	for $\alpha >0 $.
   	 	The notation $\bigohone{x}$ will be interpreted to indicate a Cauchy-null sequence with no particular power dependence of $x$.
	\end{definition}
   	 \begin{definition}
   	 	\label{def:holder}
   	 	We say that $f$ is of (point-wise) H\"older  class \holder{\beta} if for a given $x$ there exist two positive constants 
   	 	$C, \delta \in \mathbb{R} $ that for an arbitrary $ y \in Dom[ f ]$ and given $|x-y| \leq \delta$ fulfill the inequality
   	 	$
   	 	| f (x) - f (y) |  \leq C |x-y|^\beta
   	 	$, where $| \cdot |$ denotes the norm of the argument.  	
   	 
   	 \end{definition}
  	Further generalization of this concept will be given by introducing the concept of \textbf{F-analytic} functions.
  \begin{definition}
  	\label{def:fanalytic}
  	Consider a countable ordered set $\mathbb{E}^{\pm}=\left\lbrace \alpha_1 < \alpha_2 < \ldots   \right\rbrace $ of positive real constants $\alpha$.
  	Then F-analytic $\fclass{F}{E}$ is a function which is defined by the convergent (fractional) power series
  	\[
  	F(x): = c_0 + \sum\limits_{\alpha_i \in \holderspect } c_i \left( \pm x + b_i \right) ^{\alpha_i}
  	\]
  	for some sets of constants $\left\lbrace  b_i \right\rbrace $ and  $\left\lbrace  c_i \right\rbrace $.
  	The set \holderspect will be denoted further as the H\"older spectrum of $f$ (i.e. $\holderspect_f$ ).
  \end{definition}
  \begin{remark}
  	  A similar definition is used in Oldham and Spanier \cite{Oldham1974}, however, there the fractional exponents were considered to be only rational-valued for simplicity of the presented arguments.
  	  The minus sign in the formula corresponds to reflection about a particular point of interest. 
  	Without loss of generality only the plus sign convention will be assumed further. 
  \end{remark}

  	\begin{definition}
  		\label{def:deltas}
  		Define the parametrized difference operators acting on a function $f(x)$  as
  		\begin{flalign*}
	  	\deltaplus{f}{x}   & :=  f(x + \epsilon) - f(x) \ecma\\
  		\deltamin{f}{x} & :=  f(x) - f(x - \epsilon)  
  		\end{flalign*}
  		where $\epsilon>0$. The first one we refer to as \textit{forward difference} operator, 
  		the second one we refer to as \textit{backward difference} operator. 
  	\end{definition}
  	
\section{Point-wise oscillation of functions}
\label{sec:osc}
	
The concept of point-wise oscillation is used to characterize the set of continuity of a function.
\begin{definition}
	\label{def:limosc}
Define forward oscillation and its limit as the operators
 \begin{flalign*}
        	\osc{f}{x}{\epsilon}{+} : =  & \sup_{[x , x + \epsilon]} {[ f]} -
        		\inf_{[x , x + \epsilon]} {[ f]} \\
     	\mathrm{osc^{+}} [f] (x) : = & \llim{\epsilon}{0}{ \left( \sup_{[x , x + \epsilon]}   -
     		\inf_{[x , x + \epsilon]} \right) f} = \llim{\epsilon}{0}{\osc{f}{x}{\epsilon}{+}}
 \end{flalign*}
		and backward oscillation and its limit as the operators
	\begin{flalign*}
		\osc{f}{x}{\epsilon}{-} : =  & \sup_{[x - \epsilon , x ]} {[ f]} -
			\inf_{[x -   \epsilon, x ]} {[ f]} \\
	\mathrm{osc^{-}} [f] (x) : = & \llim{\epsilon}{0}{ \left( \sup_{[x - \epsilon , x ]}   -
		\inf_{[x -   \epsilon, x ]}\right)  f} = \llim{\epsilon}{0}{\osc{f}{x}{\epsilon}{-}}
	\end{flalign*}
	according to previously introduced notation \cite{Prodanov2015}.
 	\end{definition}
This definitions will be used to identify two conditions, which help characterize fractional derivatives and velocities.

   	\section{Fractional variations and fractional velocities}
   	\label{sec:frdiff}
 
  	General conditions for the existence of the fractional velocity were demonstrated in \cite{Prodanov2017}.
	It was further established that for fractional orders fractional velocity is continuous only if it is zero.
 
  \begin{definition}
  	\label{def:fracvar}
  	Define \textit{Fractional Variation} operators of order $0 \leq \beta \leq 1$ as
  	\begin{align}
  	\label{eq:fracvar1}
  	\fracvarplus {f}{x}{\beta} := \frac{ \deltaplus{f}{x}  }{\epsilon ^\beta}  
  	\\
  	\fracvarmin {f}{x}{\beta} :=  \frac{ \deltamin{f}{x} }{\epsilon ^\beta}  
  	\end{align}
  	for a positive  $\epsilon$.
  \end{definition}
   
       \begin{definition}[Fractional order velocity]
       	\label{def:frdiff}
       	Define  the \textsl{fractional velocity} of fractional order $\beta$ as the limit  
       	\begin{align}
       	\label{eq:fracdiffa}
       	\fdiffpm {f}{x}{\beta} &:= \llim{\epsilon}{0}{\frac{\Delta^{\pm}_{\epsilon} [f ] (x) }{\epsilon ^\beta}} 
       	=\llim{\epsilon}{0}{\fracvarpm {f}{x}{\beta}} \epnt       	   
       	\end{align}
       	A function for which at least one of \fdiffpm {f}{x}{\beta} exists finitely will be called $\beta$-differentiable at the point \textit{x}.
       \end{definition}
       In the above definition we do not require upfront equality of left and right $\beta$-velocities. 
       This amounts to not demanding  continuity of the   $\beta$-velocities in advance. 
       Instead, continuity is a property, which is fulfilled under certain conditions.   
     
	\begin{condition}[H\"older growth condition]
		For  given  $x$ and $0< \beta \leq 1$ 
		\begin{equation}\label{C1} 
			\mathrm{osc}_{\epsilon }^{\pm} f (x)  \leq C \epsilon^\beta \tag{C1}
		\end{equation}
		for some  $C \geq 0$ and $\epsilon > 0$.    
	\end{condition}
	\begin{condition}[H\"older oscillation condition]
		For  given  $x$, $0< \beta \leq 1$ and $\epsilon>0$ 	
		\begin{equation}\label{C2}
			\mathrm{osc}^{\pm} \fracvarpm {f}{x}{\beta} =0 \epnt \tag{C2}
		\end{equation}
	\end{condition}
     
     \begin{theorem}[Conditions for existence of $\beta$-velocity]\label{th:aexit}
     	For each $\beta > 0$ if \fdiffplus {f}{x}{\beta} exists (finitely), then $f$ is right-H\"older continuous of order $\beta$ at $x$ and \ref{C1} holds, and the analogous result holds for \fdiffmin {f}{x}{\beta} and left-H\"older continuity.
     	
     	Conversely, if \ref{C2} holds then $\fdiffpm {f}{x}{\beta}$ exists finitely.
     	Moreover, \ref{C2} implies \ref{C1}.  
     \end{theorem}
     The proof is given in \cite{Prodanov2017}.
     
 \begin{proposition}[Fractional Taylor-Lagrange property]
 	\label{th:holcomp1}
 	The existence of $\fdiffpm{f}{x}{\beta} \neq 0$  for  $\beta \leq 1$ implies that
 	\begin{equation}\label{eq:frtaylorlag}
 	f(x \pm \epsilon) = f(x) \pm \fdiffpm {f}{x}{\beta}   \epsilon^\beta + \bigoh{\epsilon^{\beta} }  \epnt
 	\end{equation}
 	While if  
 	\[
 	f(x \pm \epsilon)= f(x) \pm K \epsilon^\beta +\gamma_\epsilon \; \epsilon^\beta  
 	\] 
 	uniformly in  the interval $ x \in [x, x+ \epsilon]$ for some Cauchy sequence $\gamma_\epsilon = \bigoh{1 }$ and $K \neq 0$ is constant in $\epsilon$  then \fdiffpm {f}{x}{\beta} = K.
 \end{proposition}
The proof is given in \cite{Prodanov2017}.

\begin{remark}
	The fractional Taylor-Lagrange property was assumed and applied to establish a fractional conservation of mass formula in \cite[Sec. 4]{Wheatcraft2008} assuming the existence of a fractional Taylor expansion according to Odibat and Shawagfeh \cite{Odibat2007}.
	These authors derived fractional Taylor series development using repeated application of Caputo's fractional derivative\cite{Odibat2007}. 
\end{remark}
	\begin{proposition}
		Suppose that $f \in \fclass{F}{E}$ in the interval $I=[x, x+ \epsilon]$.
		Then \fdiffpm{f}{x}{\beta} exists finitely for $ \beta \in [0,\quad   \min{\mathbb{E}}] $.
	\end{proposition}
	\begin{proof}
		The proof follows directly from Prop. \ref{th:holcomp1} observing that
		\[
		\sum_{\alpha_j \in E \setminus \{\alpha_1 \}} a_j (x - b_j) = \bigoh{|x - b_j|^{\alpha_1} }
		\]
		so that using the notation in  Definition \ref{def:fanalytic}
		\[
		\fdiffplus{f}{x}{\alpha_1} = a_1, \quad \fdiffmin{f}{x}{\alpha_1} = 0
		\]
	\end{proof}
	Therefore, this proposition allows for characterization of an F-analytic function up to the leading fractional order
	in terms of its $\alpha$-differentiability.
	\begin{remark}
		 From the proof of the proposition  one can also see the fundamental asymmetry between the forward and  backward fractional velocities.  
		 A way to combine this is to define a complex mapping
		 \[
		 \fdiffs{\beta}{C} f (x): = \fdiffplus{f}{x}{\beta} + \fdiffmin{f}{x}{\beta} \pm i \left( \fdiffplus{f}{x}{\beta} - \fdiffmin{f}{x}{\beta} \right) 
		 \]
		 which is related to the approach taken by Nottale \cite{Nottale2010} by using complexified velocity operators.
		 However, such unified treatment will not be pursued in this work.
	\end{remark}
	   
     \section{Characterization of singular functions}
     \label{sec:sing}
     
	\subsection{Scale embedding of fractional velocities}
	\label{sec:scfracvel}
	
	As demonstrated previously, the fractional velocity has only "few" non-zero values \cite{Chen2010, Prodanov2017}. 
	Therefore, it is useful to discriminate the set of arguments where the fractional velocity does not vanish.
	
	\begin{definition}
		The set of points where the fractional velocity exists finitely and $\fdiffpm{f}{x}{\beta}  \neq 0 $ will be denoted as the \textbf{set of change}
		$ 
		\soc{\pm}{\beta} (f): = \left\lbrace x: \fdiffpm{f}{x}{\beta}  \neq 0\right\rbrace 
		$.
	\end{definition}  
	
	Since the set of change $\soc{+}{\alpha} (f)$ is totally disconnected \cite{Prodanov2017} some of the useful properties of ordinary derivatives, notably the continuity and the semi-group composition property, are lost. 
	Therefore, if we wish to retain the continuity of description we need to pursue a different approach, which should be equivalent in limit.  
	Moreover, we can treat the fractional order of differentiability (which coincides with the point-wise H\"older exponent, that is condition \ref{C1}) as a parameter to be determined from the functional expression.

	One can define two types of \textbf{scale-dependent operators}  for a wide variety if  physical signals.
	An extreme case of such signals are the singular functions, defined as
	\begin{definition}
		\label{def:singular}
		A function $f(x)$ is called \underline{singular} on the interval $x \in[a,b]$, if it is i) non-constant, 
		ii) continuous;
		iii) $f^\prime(x)=0$ Lebesgue almost everywhere (i.e. the set of non-differentiability of $f$ is of measure 0) and 
		iv) $f(a) \neq f(b)$.
	\end{definition}
	
	Since for a singular signal the derivative either vanishes or it diverges then the rate of change for such signals cannot be characterized in terms of derivatives. 
	One could apply to such signals  either the \textit{fractal variation operators}
	of certain order or the difference quotients as Nottale does and avoid taking the limit. 
	Alternately, as will be demonstrated further, the scale embedding approach can be used to reach the same goal. 
	
	Singular functions can arise as point-wise limits of continuously differentiable ones.
	Since the fractional velocity of a continuously-differentiable function vanishes we are lead to postpone taking the limit and only apply L'H\^ospital's rule, which under the hypotheses detailed further will reach the same limit as $\epsilon \rightarrow 0$.
	Therefore, we are set to introduce another pair of operators which in limit are equivalent to the fractional velocities 
	notably these are the left (resp. right) \textit{scale velocity} operators:
	\begin{flalign}
	\svarpm{f}{x}{\beta} := \frac{\epsilon^{\beta} }{ \left\lbrace \beta\right\rbrace_1  } 
	\frac{\partial}{\partial \epsilon } f (x \pm \epsilon)  
	\end{flalign}
	where $ \left\lbrace  \beta \right\rbrace_1 \equiv 1 - \beta \ \mathrm{mod} \ 1  $.
	The $\epsilon$ parameter, which is not necessarily small, represents the scale of observation.
	
	The equivalence in limit is guaranteed by the following result:
	\begin{proposition}
		\label{prop:scaledif}
		Let $f^\prime(x)$ be continuous and non-vanishing  in $(x, x \pm \mu)$. 
		Then
		\[
		\lim\limits_{\epsilon \rightarrow 0 } \fracvarpm {f}{x}{1-\beta} =  \lim\limits_{\epsilon \rightarrow 0 } \svarpm {f}{x}{\beta}
		\]
		if one of the limits exits.	
	\end{proposition}	
	The proof is given in \cite{Prodanov2016b} and will not be repeated.
	In this formulation  the value of $1-\beta$ can be considered as the magnitude of deviation from  linearity (or differentiability) of the signal at that point.

	\begin{theorem}[Scale velocity fixed point]
	\label{th:scalerec}
	Suppose that $f  \in BVC[x, x+ \epsilon]$ and $f^\prime $ does not vanish in $[x, x+ \epsilon] $.
	Suppose that $\phi \in \fclass{C}{1}$ is a contraction map. 
	Let $f_n (x) := \underbrace{  \phi  \circ \ldots \phi}_{n} \circ f(x) $ be the $n$-fold composition and 
	$F(x) := \llim{n}{\infty}{f_n}(x) $.
	Then the following commuting diagram holds:
	
	\begin{center}
		\begin{tikzpicture}\matrix (m) [matrix of math nodes, row sep=3em, column sep=5em, minimum width=2em, minimum height=2em, nodes={asymmetrical rectangle}]
		{
			f_n (x) &  \svarpm{f_n}{x}{1-\beta} \\
			F (x)		&  \fdiffpm{F}{x}{\beta}  \\
		};	
		\path[  -> ]
		(m-1-1) edge node[above] {$ \svars{1-\beta}{\epsilon \pm } $} 
		(m-1-2)
		edge node [right]{$\lim_{n \rightarrow \infty}  $ } (m-2-1);
		\path[  -> ]
		(m-1-2)  edge node[right] {$\lim_{n \rightarrow \infty}  $ } node[left]{$\epsilon_n$} (m-2-2);
		\path[  -> ]
		(m-2-1)  edge node[above] {$\fdiffs{\beta}{\pm} $ }   (m-2-2);
		\end{tikzpicture}
	\end{center} 
	The limit in $n$ is taken point-wise. 	 
	\end{theorem}
	\begin{proof}
	The proof follows by induction.
	Only the right case will be proven. The left case follows by reflection of the function argument.
	Consider an arbitrary $n$ and an interval $I=[x, \, x+\epsilon]$.
	By differentiability of the map $\phi$
	\[
	\fdiffplus{f_n}{x}{\beta} = \left( \frac{\partial \phi}{\partial f} \right) ^n \fdiffplus{f}{x}{\beta} 
	\]
	Then by hypothesis $f^{\prime} (x)$ exists finitely  a.e. in \textit{I} so that
	\[
	\fdiffplus{f}{x}{\beta} = \frac{1}{\beta}\llim{\epsilon}{0}{ \epsilon^{1-\beta}  f^{\prime} (x+ \epsilon)}
	\]
	so that
	\[
	\fdiffplus{f_n}{x}{\beta} = \frac{1}{\beta} \left( \frac{\partial \phi}{\partial f} \right) ^n \llim{\epsilon}{0}{ \epsilon^{1-\beta}  f^{\prime} (x+ \epsilon)}
	\]
	On the other hand
	\[
	\svarplus{f_n}{x}{1- \beta} =  \frac{1}{\beta} \left( \frac{\partial \phi}{\partial f} \right) ^n   \epsilon^{1-\beta}  f^{\prime} (x+ \epsilon)
	\]
	Therefore, if the RHS limits exist they are the same. 
	Therefore, the equality is asserted for all \textit{n}.
	
	Suppose that \fdiffplus{F}{x}{\beta} exists finitely.
	
	Since $\phi$ is a contraction and $F(x)$ is its fixed point by Banach fixed point theorem there is a   Lipschitz constant $q<1$ such that
	\begin{flalign*}
		\underbrace{|F(x+ \epsilon) - f_n (x+ \epsilon) |}_{A} & \leq \frac{q^n}{1-q} | \phi \circ f  (x+ \epsilon) - f  (x+\epsilon) | \\
		\underbrace{|F(x) - f_n (x) |}_{B} & \leq \frac{q^n}{1-q} | \phi \circ f  (x) - f  (x) | \\
	\end{flalign*}
	Then  by the triangle inequality
	\[
	| \deltaplus{F}{x} -  \deltaplus{f_n}{x} | \leq A + B \leq \frac{q^n}{1-q} \left(  | \phi \circ f  (x+ \epsilon) - f  (x+\epsilon) | + | \phi \circ f  (x) - f  (x) | \right) \leq q^n L
	\]
	for some $L$.
	Then
	\[
	| \fracvarplus{F}{x}{\beta} - \fracvarplus{f_n}{x}{\beta} | \leq \frac{q^n L}{\epsilon^\beta}
	\]
	We evaluate $\epsilon =   q^n / \lambda$ for some $\lambda \geq 1$ so that
	\[
	| \fracvarplus{F}{x}{\beta} - \fracvarplus{f_n}{x}{\beta} | \leq  q^{n (1-\beta)} L \lambda^\beta  =: \mu_n
	\]
	Therefore, in limit RHS $ \llim {n} { \infty} {\mu_n}=0$.
	Therefore,
	\[
	\llim{n}{\infty}{ \fracvarplus{f_n}{x}{\beta}  } =\fdiffplus{F}{x}{\beta} \quad | \quad \epsilon_n \rightarrow 0
	\]
	\end{proof}
	So stated, the theorem holds also for sets of maps $\phi_k $ acting in sequence as they can be identified with an action of a single map.
	\begin{corollary}
	Let $\Phi= \{\phi_k  \} $, where the domains of $\phi_k $ are disjoint and the hypotheses of  Theorem \ref{th:scalerec} hold.
	Then Theorem \ref{th:scalerec} holds for $\Phi$.
	\end{corollary}
	\begin{corollary}
		Under the hypotheses of  Theorem \ref{th:scalerec} for $f \in \holder{\alpha}$ and $ \frac{\partial \phi}{\partial f} >1 $
		 there is a  Cauchy null sequence $ \{ \epsilon \}_k^{\infty}$, such that
		 \[
		 \llim{n}{\infty}{ \svars{\epsilon_n \pm }{1- \alpha} f_n (x)  } =\fdiffpm{f}{x}{  \alpha}  
		 \]
		 This sequence will be named \textbf{scale--regularizing} sequence.
	\end{corollary}
	\begin{proof}
		In the proof of the theorem it was established that
		\[
		\svarplus{f_n}{x}{1- \alpha} =  \frac{1}{\beta} \left( \frac{\partial \phi}{\partial f} \right) ^n   \epsilon^{1-\alpha}  f^{\prime} (x+ \epsilon)
		\]
		Then   we can identify
		\[
		\epsilon_n = \pm  \epsilon 
		  \left/ \left|  \frac{\partial \phi}{\partial f}  \right| ^{n/(1-\alpha) }    \right.                   
		\]
		so that
		\[
		\svarplus{f_n}{x}{1- \alpha}  = \svars{\epsilon_n }{1- \alpha} [f](x)
		\]

		Then since 	$ \frac{\partial \phi}{\partial f} >1 $ we have $ \epsilon_{n+1}< \epsilon_n <1 $. 
		Therefore, $ \{ p \, \epsilon_k\}_k$ is a Cauchy sequence for a positive  number \textit{p}.
		Further, the RHS limit evaluates to (omitting $n$ for simplicity)
		\[
		\llim{\epsilon}{0}{ \svarpm{f}{x}{1- \alpha}  } = 
		\frac{1}{\alpha}\llim{\epsilon}{0}{\epsilon^{1-\alpha} f^{\prime} (x \pm \epsilon)} = \fdiffpm{f}{x}{\alpha}
		\] 
		The backward case can be proven by identical reasoning. 
	\end{proof}
	Therefore, we can identify the Lipschitz constant $q$ by the properties of the contraction maps as will be demonstrated in the following examples. 
	
	\subsection{Examples}
     \label{sec:examples}
   
   		\begin{description}
   			\item[De Rham-Neidinger functions]

   	 De Rham's function is defined by the functional equations
     \[
     R_a(x):=
     \begin{cases} 
     a  R_a(2 x) ,&  0  \leq x < \frac{1}{2}  \\
     (1-a)  R_a(2 x-1) +a ,&  \frac{1}{2} \leq x \leq 1 \\
     \end{cases}  
     \]
     and boundary values $R_a(0)=0$, $R_a(1)=1$. the function is strictly increasing and singular. 
     It is also is known under several different names --
     "Lebesgue's singular function"  or "Salem's singular function". 
     This function  was also defined in different ways \cite{Cesaro1906,Faber1909}.
     Lomnicki and Ulan \cite{Lomnicki1934}, for example, give a probabilistic construction as follows.
     In a an imaginary experiment of flipping a possibly "unfair" coin with probability $a$ of heads (and 
     $1 - a$ of tails). Then
     $
     R_a(x) = \mathbb{P} \left\lbrace t \leq x \right\rbrace 
     $
     after infinitely many trials  where $t$ is the record of the trials represented by a binary sequence.  
     While Salem \cite{Salem1943} gives a geometrical construction.
     
    $r_{0}(x, a)= x^{a}$, which verifies $ r_{0}(0, a)=0$ and  $ r_{0}(1, a)=1$. 
    \[
    r_n(x,a):= \left\{
    \begin{array}{ll}
    \frac{1}{2^a}  r_{n-1}(2 \, x, a) ,&  0  \leq x < \frac{1}{2}  \\
    (1-\frac{1}{2^a})  r_{n-1}(2 \, x-1, a) + \frac{1}{2^a} ,&  \frac{1}{2} \leq x \leq 1 \\
    \end{array} \right.
    \]
    The system is a slight re-parametrization of the original De Rham's function.
	   Then
    $
    R_a(x) =\llim{n}{\infty}{r_{n}(x,a)}  
    $.
    
    Formal calculation shows that
    \[
    \fdiffs{\beta}{+} r_n(x,a):= \left\{
    \begin{array}{ll}
     2^{\beta -a} \fdiffs{\beta}{+} r_{n-1}(2 \, x, a) ,&  0  \leq x < \frac{1}{2}  \\
    (1-\frac{1}{2^a}) 2^\beta \fdiffs{\beta}{+}  r_{n-1}(2 \, x-1, a) ,&  \frac{1}{2} \leq x \leq 1 \\
    \end{array} \right.
    \]
    Therefore, for $\beta < a$ the velocity vanishes, while for $\beta > a$ it diverges.
    We further demonstrate its existence for $\beta=a$. 
    For this case
    \[
    \fdiffs{a}{+} r_n(x,a):= \left\{
    \begin{array}{ll}
     \fdiffs{a}{+} r_{n-1}(2 \, x, a) ,&  0  \leq x < \frac{1}{2}  \\
    (2^a -1 )   \fdiffs{a}{+}  r_{n-1}(2 \, x-1, a) ,&  \frac{1}{2} \leq x \leq 1 \\
    \end{array} \right.
    \]
    and $\fdiffs{a}{+} r_{0}(x,a)= \mathbf{1} (x=0) $.
    
    The same result can be reached using scaling arguments. 
   	We can discern two cases.
   	
   	Case 1, $x \leq 1/2 $ : Then application of the scale  operator leads to :
   	\[ 
   	\svarplus {r_{n}}{x,a}{1-a} = \frac{  2^{1-a} \epsilon^{1-a}  }{a}   \frac{\partial }{\partial \epsilon }r_{n-1}  (2 x+2 \epsilon, a) 
   	\]
  
   	Therefore, the factor will remain scale invariant for  $\epsilon=1/2$ and consecutively  
   	$
   	\epsilon_n =   \frac{1}{2^{n  }} 
   	$ so that  we identify a scale-regularizing Cauchy sequence so that $f(x)=x^a$ is verified and $  \fdiffs{a}{+} R_a (0 ) =1 $.
   	
   	Case 2, $x > 1/2$:  In a similar way :
   	\[ 
   	\svarplus {r_{n}}{x,a}{1-a} = \frac{2^{1-a}  \epsilon^{1-a}}{a}  \left(  2^a -1 \right)   \frac{\partial }{\partial \epsilon } r_{n-1}(2 x + 2 \epsilon -1, a)
   	\]
   	Applying the same sequence results in a factor  $2^a - 1 \leq 1$. Therefore, the resulting transformation is a contraction. 
   	
   	Finally, we observe that if $ x = \overline{0.d_1 \ldots d_n} $ then the number of occurrences of Case 2 corresponds to the amount of the digit $1$ in the binary notation, that is to 
   	$s_n = \sum\limits_{k=1}^{n} d_k $.
   	
 The calculation can be summarized in the following proposition:
 \begin{proposition}
 	\label{prop:derham1}
 	
 	Let $\mathbb{Q}_2$ denote the set of dyadic rationals. 
 	Let  $s_n = \sum\limits_{k=1}^{n} d_k $ denote the sum of the digits for the number  
 	$x =  \overline{0.d_1 \ldots d_n}, \ d \in \{0, 1 \}$ in  binary representation, 
 	then
 	\[
 	\fdiffplus{R_a}{x}{\beta}= \left\{  
 	\begin{array}{ll} \left( 2^\beta -1 \right)^{s_{n}-1}, & x \in \mathbb{Q}_2 \\
 	0, & x \notin  \mathbb{Q}_2
 	\end{array}
 	\right.
 	\]
 	for $\beta= - log_2 a$, $a \neq \dfrac{1}{2}$.
 	For $\beta < - log_2 a$ 
 	$\fdiffplus{R_a}{x}{\beta}=0$.
 \end{proposition}
     
     Neidinger introduces a novel strictly singular function \cite{Neidinger2016}, 
      called f\textit{air-bold gambling function} and which will be referred to as Neidinger's function. 
     The function is based on De Rham's construction.
     The Neidinger's function is defined as the limit of the system
      \[
     N_n(x, a):=
     \begin{cases} a\leftarrow  1-a, & n \quad even \\
      a  N_{n-1}(2 x, a) ,&  0  \leq x < \frac{1}{2}  \\
     (1-a) \, N_{n-1}(2 x-1, a) + a ,&  \frac{1}{2} \leq x \leq 1 \\
     \end{cases}  
     \]
     where $N_a(x,0)=x$. 
     In other words the parameter $a$ alternates for every recursion step.
  
     We can exercise a similar calculation again starting from $r_0(x, a)=x^a $.
     Then
     \[
     \fdiffs{\beta}{+} r_n(x,a):= \left\{
     \begin{array}{ll}
	    a\leftarrow  1-a, & n \quad even \\
	    a 2^{\beta } \fdiffs{\beta}{+} r_{n-1}(2 \, x, a) ,&  0  \leq x < \frac{1}{2}  \\
	    (1-a) 2^\beta \fdiffs{\beta}{+}  r_{n-1}(2 \, x-1, a) ,&  \frac{1}{2} \leq x \leq 1 \\
     \end{array} \right.
     \]
     Therefore, either $a=1/2^\beta$ or $1-a = 1/2^\beta$ so that $\fdiffs{a}{+} r_{0}(x,a)= \mathbf{1} (x=0) $.
     The velocity can be computed algorithmically to arbitrary precision.  
     
     \begin{figure}
     	\centering
     \includegraphics[width=0.7\linewidth ]{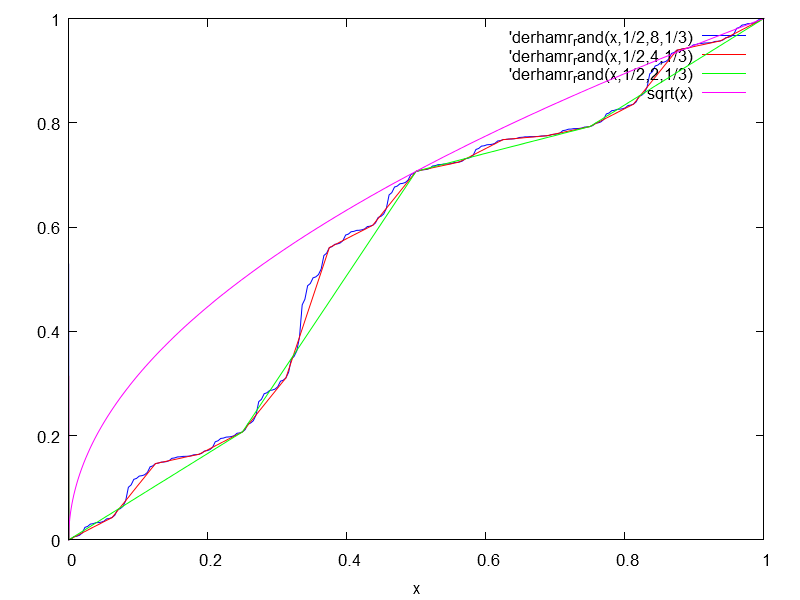}
     	\caption{Neidinger's function}
     	Recursive construction of the Neidinger's function; iteration levels 2,4,8.
     	\label{fig:derham1s}
     \end{figure}
     \begin{figure}
     	\centering
     	\includegraphics[width=0.7\linewidth]{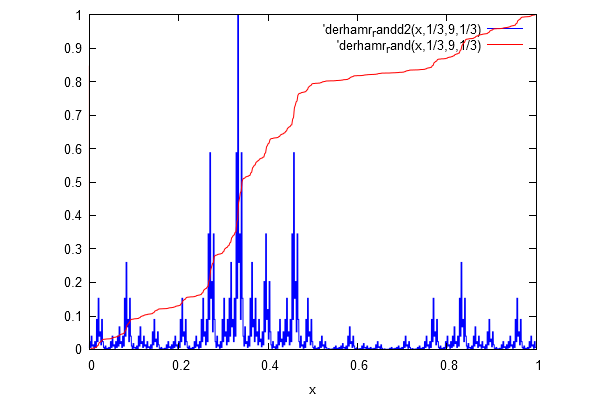} \\
     	\includegraphics[width=0.7\linewidth]{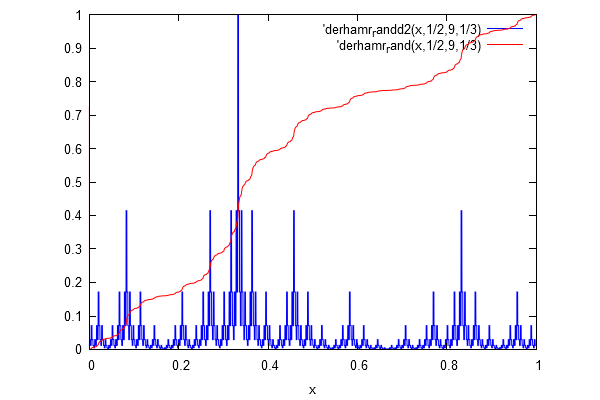}
     	\caption{Fractional velocity of Neidinger's function}	\label{fig:derham1d}
     	Recursive construction of the fractional velocity for 
     	$\beta=1/3$ and $\beta=1/3$, iteration level 9. 
     	The Neidinger's function IFS are given for comparison for the same iteration level.
     \end{figure}

   \item[Langevin evolution]     
   Consider a non-linear problem, where the phase-space trajectory of a system is represented by a F-analytic function $x(t)$ and $t$ is a real-valued parameter, for example time or curve length.
   That is, suppose that a generalized Langevin equation holds: 
   \[
   \deltaplus{x}{t}     = a(x, t) \epsilon + B(x, t) \epsilon^\beta + \bigoh{\epsilon}, \ \beta \leq 1
   \]
   The form of the equation depends critically on the assumption of continuity of the reconstructed trajectory. 
   This in turn demands for the fluctuations of the fractional term  to be discontinuous.
   The proof technique is introduced in \cite{Prodanov2017}, while the argument is similar to the one presented in \cite{Gillespie1996}. 
   
   By hypothesis  $\exists K$, such that $|\Delta_\epsilon X| \leq K \epsilon^\beta $ and $x(t)$ is \holder{\beta}.
   Therefore, without loss of generality we can set $a=0$ and apply the argument from \cite{Prodanov2017}.
   Fix the interval $[t, t+ \epsilon]$ and choose a partition of points $\{t_k= t +k/N \epsilon \}$ for an integral $N$.
   \[
   x_{t_k} = x_{t_{k-1}} + B (x_{t_{k-1}} , t_{k-1}) \left( \epsilon/N\right) ^\beta + \bigoh{ \epsilon^\beta}
   \]
   where we have set $x_{t_k} \equiv x({t_k})$.
   Therefore,
    \[
   \Delta_\epsilon  x = \frac{1}{N^\beta}\sum\limits_{ k=0}^{N-1}  B (x_{t_{k}} , t_k) \epsilon  ^\beta + \bigoh{ \epsilon^\beta}
   \]
Therefore, if we suppose that $B$ is continuous in $x$ or $t$ after taking limit on both sides we arrive at
   \[
   \limsup_{\epsilon \rightarrow 0} \frac{\Delta_\epsilon  x}{\epsilon  ^\beta} = B (x_t, t) = 
   \frac{1}{N^\beta}\sum\limits_{ k=0}^{N-1} \limsup_{\epsilon \rightarrow 0}  B (x_{t_{k}} , t_k)  =
    N^{1-\beta} B (x_t, t) 
   \]
   so that $(1-  N^{1-\beta}) B (x, t) =0 $.
   Therefore, either $\beta=1$ or $B (x, t) =0$. So that  $B(x, t)$ must oscillate and is not continuous if $\beta <1$.     
  
   	\end{description}  
  
   \section{Characterization of Kolwankar-Gangal local fractional derivatives}
   \label{sec:KG}
   
    The overlap of the definitions of the Cherebit's fractional velocity and the Kolwankar-Gangal fractional derivative is not complete
    \cite{Adda2013}.
    Notably, Kolwankar-Gangal fractional derivatives are sensitive to the critical local H\"older exponents, while the fractional velocities are sensitive to the critical point-wise H\"older exponents and there is no complete equivalence between those quantities \cite{Kolwankar2001}.
   In this section we will characterize the local fractional derivatives in the sense of Kolwankar and Gangal using the notion of fractional velocity.

   \subsection{Fractional integrals and derivatives}
   \label{sec:fi}
   The left Riemann-Liouville differ-integral of order $\beta \geq 0$ is defined as
   \[
   \frdiffiix{\beta}{ a+ }  f (x) = 
   \dfrac{1}{\Gamma(\beta)} \int_{a}^{x}   f \left( t \right)  \left( x-t \right)^{\beta -1}dt 
   \]
   while the right integral is defined as
   \[
   \frdiffiix{\beta}{ -a }  f (x) = 
   \dfrac{1}{\Gamma(\beta)} \int_{x}^{a}   f \left( t \right)  \left( t-x \right)^{\beta -1}dt 
   \]
   where $\Gamma(x) $ is the Euler's Gamma function  (Samko et al. \cite{Samko1993} [p. 33]). 
   The left (resp. right) Riemann-Liouville (R-L) fractional derivatives are defined as the expressions (Samko et al. \cite{Samko1993}[p. 35]):
   \begin{flalign*}
   \mathcal{D}_{a+}^{\beta} f (x)  & := \frac{d}{dx} \frdiffiix{1-\beta}{ a+ }  f (x)  = \frac{1}{\Gamma(1- \beta)}  \frac{d}{dx}  \int_{a}^{x}\frac{  f (t ) }{{\left( x-t\right) }^{\beta }}dt  \\
   \mathcal{D}_{-a}^{\beta} f (x) & := -\frac{d}{dx}  \frdiffiix{1-\beta}{ -a }  f (x) = -\frac{1}{\Gamma(1- \beta)} \frac{d}{dx}  \int_{x}^{a}\frac{  f (t )  }{{\left( t-x\right) }^{\beta }}dt
   \end{flalign*}
      
   The left (resp. right) R-L derivative of a function $f$ exists for functions  representable by fractional integrals of order $\alpha$ of some Lebesgue-integrable function.
   This is the spirit of the definition of Samko et al. \cite{Samko1993}[Definition 2.3, p. 43] for the appropriate functional spaces: 

   \begin{flalign*}
   \mathcal{I}^{\alpha}_{a, +} (L^1)& := \left\lbrace f: 
   \frdiffiix{ \alpha}{ a+ }  f (x) \in AC([a, b]), f  \in L^1 ([a,b]) , x \in [a, b] \right\rbrace \ecma  \\
   \mathcal{I}^{\alpha}_{a, -} (L^1)& := \left\lbrace f: 
   \frdiffiix{ \alpha}{ -a }  f (x) \in AC([a, b]), f  \in L^1 ([a,b]) , x \in [a, b] \right\rbrace 
   \end{flalign*}
   Here $AC$ denotes absolute continuity on an interval in the conventional sense. 
   Samko et al. comment that the existence of a summable derivative $f^\prime(x)$ of a function $f( x)$ does not yet guarantee
   the restoration of $f(x)$ by the primitive in the sense of integration and go on to give references to singular functions for which
   the derivative vanishes almost everywhere and yet the function is not constant, such as for example, the De Rhams's function \cite{Rham1957}. 
   
   To ensure restoration of the primitive by fractional integration, based on Th. 2.3 Samko et al. introduce another space of \textit{summable fractional derivatives}, for which the Fundamental Theorem of Fractional Calculus 
   holds. 
  \begin{definition}
     	\label{def:summdfderiv}
     	Define the functional spaces of summable fractional derivatives \\ (Samko et al. \cite{Samko1993}[Definition 2.4, p. 44])  as
     	$
     	E^{\alpha}_{a,\pm} ([a, b]):= \left\lbrace f: 
     	\mathcal{I}^{1-\alpha}_{a, \pm} (L^1)
     	\right\rbrace 
     	$.
  \end{definition}
   In this sense 
     \[
     \frdiffiix{ \alpha}{ a+ }  \left(   \mathcal{D}_{a+}^{\alpha} \, f  \right) (x)  = f (x) 
     \]
     for $f \in E^{\alpha}_{a, +} ([a, b])$ (Samko et al. \cite{Samko1993}[Th. 4, p. 44]).
   While 
    \[
    \mathcal{D}_{a+}^{\alpha} \left( \frdiffiix{ \alpha}{ a+ } \, f \right) (x)  = f (x) 
    \]
    for $f \in  \mathcal{I}^{\alpha}_{a, +} (L^1)$.

   So defined spaces do not coincide.
   The distinction can be seen from the following example:
   \begin{example}
	Define \[
	h(x):= \begin{cases}
			0, & x \leq 0 \\
			x^{\alpha -1}, & x>0
	\end{cases}
	\]
	for $0 < \alpha < 1$.
	Then 
	$  \frdiffiix{1- \alpha}{ 0+ }  h(x) = \Gamma (\alpha) $ for $x>0$ 	so that $ \mathcal{D}_{0+}^{\alpha} \,  h (x) =0$ everywhere in \fclass{R}{}.
	
	On the other hand,
	$$ 
	\frdiffiix{ \alpha}{ 0+ }  h(x) = \frac{ \Gamma (\alpha)}{\Gamma (2 \alpha)} \ x^{2 \alpha- 1} $$ for $x>0$
	and
	$$
	\frac{ \Gamma (\alpha)}{\Gamma (2 \alpha)} \  \mathcal{D}_{0+}^{\alpha} x^{2 \alpha- 1} = x^{\alpha -1} \epnt
	$$
	Therefore, the fundamental theorem fails.
	It is easy to demonstrate that $h(x)$ is not $AC$ on any interval involving $0$.
   \end{example}
   Therefore, the there is an inclusion $  E^{\alpha}_{a, +} \subset \mathcal{I}^{\alpha}_{a, +} $.
   \subsection{The local(ized) fractional derivative}
   \label{sec:lfd}
   
   The definition of \textsl{local fractional derivative} (LFD)  introduced by Kolwankar and Gangal \cite{Kolwankar1997} is based on the localization of Riemann-Liouville fractional derivatives towards a particular point of interest in a way similar to Caputo.   
   \begin{definition}
   	\label{def:kg-lfd}
   	Define left LFD as
   	\[ 
   	\mathcal{D}_{KG+}^{\beta}  f(x) := \llim{x}{a}{} \mathcal{D}_{a+}^{\beta} \left[  f-f(a) \right] (x) 
   	\]
   	and right LFD as
   	\[ 
   	\mathcal{D}_{KG-}^{\beta}  f(x) :=  \llim{x}{a}{} \mathcal{D}_{-a}^{\beta}  \left[  f(a) - f  \right] (x) \epnt 
   	\] 
   \end{definition}
   \begin{remark}
   The seminal publication defined only the left derivative. 
   Note that the LFD is more restrictive than the R-L derivative  because the latter may not have a limit as $x \rightarrow a$.	
   \end{remark}
   
   Ben Adda and Cresson \cite{Adda2001} and later Chen et al. \cite{Chen2010} claimed that the Kolwankar -- Gangal  definition of local fractional derivative is equivalent to Cherbit's definition for certain classes of functions.
   On the other hand, some inaccuracies can be identified in these articles \cite{Chen2010, Adda2013}. 
   Since the results of Chen et al. \cite{Chen2010} and Ben Adda-Cresson  \cite{ Adda2013} are proven under different hypotheses and notations I feel that separate proofs of the equivalence results using the theory established so-far are in order.
   
   \begin{proposition}[LFD equivalence]
   	\label{prop:ldfeq}
   	Let $f(x)$ be $\beta$-differentiable about $x$. Then $\mathcal{D}_{KG, \pm }^{\beta} f(x)$ exists and
   	\[
   	\mathcal{D}_{KG, \pm }^{\beta} f(x)  = \Gamma(1+\beta) \ \fdiffpm {f}{x}{\beta} \epnt
   	\]
   \end{proposition}
   \begin{proof}
   	We will assume that $f(x)$ belongs to \holder{r,\beta} and is non-decreasing in the interval $ [a, a+x]$.
   	Since $x$ will vary, for simplicity let's assume that $\fdiffplus {f}{a}{\beta} \in \soc{}{\beta}$.
   	Then by Prop. \ref{th:holcomp1} we have
   	\[
   	f(z ) = f(a) + \fdiffplus {f}{a}{\beta} (z - a)^\beta + \bigoh{(z - a)^\beta}, \ a \leq z \leq x \epnt
   	\]
   	Standard treatments of the fractional derivatives \cite{Oldham1974} and the changes of variables $u=(t-a)/(x-a)$ give the alternative Euler integral formulation
   	\begin{equation}
   	\label{eq:eulerint}
   	\mathcal{D}_{+a}^{\beta} f(x) = \frac{\partial}{\partial h} \left( \frac{ h^{1-\beta}}{\Gamma(1- \beta) }  \ \int\limits_0^1 \frac{ f( h u +a ) - f(a)}{(1-u)^{\beta} } du \right) 
   	\end{equation}
   	for $ h=x-a$.
   	Therefore, we can evaluate the fractional  Riemann-Liouville integral as follows:
   	\[
   	\frac{ h^{1-\beta} }{\Gamma(1- \beta) }    \int\limits_0^1 \frac{ f( h u +a ) - f(a)}{(1-u)^{\beta} } du= 
   	\frac{ h^{1-\beta} }{\Gamma(1- \beta) }    \int\limits_0^1 \frac{ K  \left(  h  u\right) ^\beta  + \bigoh{(h u)^\beta}  }{(1-u)^{\beta} } du =: I
   	\]
   	setting conveniently $K= \fdiffplus {f}{a}{\beta} $. 
   	The last expression $I$ can be evaluated in parts as 
   	\[
   	I =  \underbrace{\dfrac{ h^{1-\beta} }{\Gamma(1- \beta) }    \int\limits_0^1 \frac{ K h^\beta u^\beta   }{(1-u)^{\beta} } du }_A +
   	\underbrace{\dfrac{ h^{1-\beta} }{\Gamma(1- \beta) }    \int\limits_0^1 \frac{   \bigoh{(h u)^\beta}  }{(1-u)^{\beta} } du}_C  \epnt
   	\]
   	The first expression is recognized as the Beta integral \cite{Oldham1974}:
   	\[
   	A = \frac{ h^{1-\beta} }{\Gamma(1- \beta) } B \left(1-\beta, 1+\beta \right) h^\beta  K = \Gamma(1+\beta) \, K h
   	\] 
   	In  order to evaluate the second expression we observe that 
   	by Prop. \ref{th:holcomp1}  
   	\[\left|  \bigoh{(h u)^\beta} \right|   \leq \gamma (h u)^\beta \]  
   	for a  positive $\gamma = \bigohone{1} $. 
   	Assuming without loss of generality that $f(x)$ is non decreasing in the interval we have
   	$
   	C \leq \Gamma(1+\beta) \, \gamma h
   	$ and 
   	\[
   	\mathcal{D}_{a+}^{\beta} f(x) \leq \left( K + \gamma\right) \Gamma(1+\beta)
   	\]
   	and the limit gives
   	$\llim{x}{a+}{K + \gamma} = K$ by the \textit{squeeze lemma} and Prop. \ref{th:holcomp1}.
   	Therefore, 
   	$
   	\mathcal{D}_{KG + }^{\beta} f(a)  = \Gamma(1+\beta) \fdiffplus {f}{a}{\beta} 
   	$.
   	On the other hand, for   $\holder{r,\alpha}$ and $\alpha > \beta$ by the same reasoning 
   	\[
   	A = \frac{ h^{1-\beta} }{\Gamma(1- \beta) } B \left(1-\beta, 1+\alpha \right) h^\alpha  K = \Gamma(1+\beta) \, K h^{1- \beta+ \alpha } \epnt
   	\]
   	Then differentiation by \textit{h} gives
   	\[ 
   	A^\prime_h=  \frac{\Gamma(1+\alpha)}{\Gamma(1+\alpha -\beta)} \, K h^{ \alpha - \beta } \epnt
   	\]
   	Therefore,
   	\[
   	\mathcal{D}_{KG+}^{\beta} f(x) \leq \frac{\Gamma(1+\alpha)}{\Gamma(1+\alpha -\beta)} \left( K + \gamma\right)   h^{ \alpha - \beta } 
   	\] 
   	by monotonicity in $h$.
   	Therefore, 
   	$
   	\mathcal{D}_{KG \pm }^{\beta} f(a)  = \fdiffpm {f}{a}{\beta} = 0 
   	$.
   	Finally, for $\alpha =1$ the expression $A$ should be evaluated as the limit $\alpha \rightarrow 1$ due to divergence of the $\Gamma$ function.
   	The proof for the left LFD follows identical reasoning, observing the backward fractional Taylor expansion property.   	
   \end{proof}
   
   \begin{proposition}
   	\label{prop:KGconverse2}
   	Suppose that $\mathcal{D}_{KG \pm}^{\beta}  f(x) $ exists finitely and the related R-L derivative is summable in the sense of Def. \ref{def:summdfderiv}.
   	Then   $f$ is $\beta$-differentiable about $x$
   	and $ 	\mathcal{D}_{KG, \pm }^{\beta} f(x)  = \Gamma(1+\beta) \ \fdiffpm {f}{x}{\beta}    	$.
   \end{proposition}
   \begin{proof}
   	Suppose that $f \in  E^{\alpha}_{a,+} ([a, a + \delta])$ and let $\mathcal{D}_{KG +}^{\alpha}  f(x)=L $. 
   	The existence of this limit implies the inequality
   	\[
   	\left|  \mathcal{D}_{a+}^{\alpha} \left[ f-f(a) \right] (x)  - L \right| < \mu  
   	\] for $|x-a|  \leq  \delta$ and a Cauchy sequence $\mu$. 
   	
   	Without loss of generality suppose that $\mathcal{D}_{a+}^{\alpha} \left[  f-f(a) \right] (x) $ is non-decreasing and $L \neq 0$.
   	We proceed by integrating the inequality:
   	\[
   	\frdiffiix{ \alpha}{ a+ } \ \left( \mathcal{D}_{a+}^{\alpha} \left[ f-f(a) \right] (x)  - L  \right)  < \frdiffiix{ \alpha}{ a+ } \mu 
   	\]
   	Then by the Fundamental Theorem
   	\[
   	f (x) - f(a) -  \frac{ L }{\Gamma(\alpha)} (x-a)^\alpha < \frac{ \mu (x-a)^\alpha}{\Gamma(\alpha)}  
   	\]
   	and
   	$$
   	\frac{ 	f (x) - f(a) - L/\Gamma(\alpha) }{ (x-a)^\alpha} < \frac{\mu}{\Gamma(\alpha)} = \bigoh {1} \epnt
   	$$
   	Therefore, by  Prop \ref{th:holcomp1}  $f$ is $\alpha$-differentiable at \textit{x} and
   	$\mathcal{D}_{KG, + }^{\alpha} f(x)  = \Gamma(1+\alpha) \ \fdiffplus{f}{x}{\alpha}$.
   	The last assertion comes from Prop. \ref{prop:ldfeq}.
   	The right case can be proven in a similar manner. 
   \end{proof}
   
   The weaker condition of only point-wise H\"older continuity requires the additional hypothesis of summability as identified in \cite{Adda2013}. The following results can be stated.
   
   \begin{lemma}
   	\label{th:KGconverse1}
   	Suppose that $\mathcal{D}_{KG \pm}^{\beta}  f(a) $ exists finitely in the weak sense, i.e. implying only that $f \in 
 \mathcal{I}^{\alpha}_{a, +} (L^1)$. Then condition \ref{C1} holds for $f$ a.e.  in the  interval $[a, x +\epsilon]$.
   \end{lemma}
   \begin{proof}
   	The left R-L derivative can be evaluated as follows. 
   	Consider the fractional integral in the Liouville form
   	\begin{flalign*}
   	I_1 &=   \int\limits_{0}^{\epsilon+ x-a} \frac{ f(x+ \epsilon - h) - f(a)}{h^\beta}\, d h -
   	\int\limits_{0}^{  x-a} \frac{ f(x - h) - f(a)}{h^\beta}\, d h  \\
   	& =  \underbrace{  \int\limits_{x-a}^{\epsilon+ x-a} \frac{ f(x+ \epsilon - h) - f(a)}{h^\beta}\, d h}_{I_2}
   	+ \underbrace{  \int\limits_{0}^{x-a} \frac{ f(x +\epsilon - h) - f(x-h)}{h^\beta}\, d h}_{I_3}
   	\end{flalign*}
   	Without loss of generality assume that $f$ is non-decreasing in the interval $[a, x + \epsilon - a]$ and 
   	set
   	$ M_{y, x} = \sup_{[x, y]} f  - f(x)$ and $  m_{y, x} = \inf_{[x, y]} f - f(x)$.
   	Then  
   	\[
   	I_2 \leq  \int\limits_{x-a}^{\epsilon+ x-a} \frac{ M_{x+ \epsilon , a} }{h^\beta}\, d h = 
   	\frac{M_{x+ \epsilon , a}}{1 -\beta}  \left[  { {\left( x- \epsilon + a\right) }^{1-\beta} -\left( x-a\right) }^{1-\beta} \right] \leq   \epsilon \frac{M_{x+ \epsilon , a}  }{\left( x-a\right)^\beta } + \bigoh{\epsilon^2}
   	\]
   	for $x \neq a$.
   	In a similar manner
   	$$
   	I_2  \geq m_{x+ \epsilon , a} \frac{\epsilon }{\left( x-a\right)^\beta } + \bigoh{\epsilon^2} \epnt
   	$$ 
   	Then dividing by $\epsilon$ gives
   	\[
   	\frac{  m_{x+ \epsilon , a}}{\left( x-a\right)^\beta } + \bigoh{\epsilon } \leq \frac{I_2}{\epsilon} \leq   \frac{  M_{x+ \epsilon , a}}{\left( x-a\right)^\beta } + \bigoh{\epsilon } 
   	\]
   	Therefore, the quotient limit is bounded from both sides as
   	\[
   	\frac{ m_{x , a}}{ \left( x-a\right)^\beta }\leq   \underbrace{ \llim{\epsilon}{0}{\frac{I_2}{\epsilon}} }_{I_2^\prime} \leq \frac{  M_{x , a}}{ \left( x-a\right)^\beta }
   	\]
   	by the continuity of $f$. 
   	In a similar way we establish
   	\begin{flalign*}
   	I_3  \leq  \int\limits_{0}^{x-a} \frac{ M_{x +\epsilon , x} }{h^\beta} \, d h  = \frac{  M_{x +\epsilon , x} }{1-\beta}\left( x-a \right)^{1-\beta} 
   	\end{flalign*}
   	and
   	\[
   	\frac{  m_{x +\epsilon , x} }{1-\beta}\left( x-a \right)^{1-\beta} \leq I_3 
   	\]
   	Therefore,
   	\[
   	\frac{  m_{x +\epsilon , x} }{\left( 1-\beta \right) \epsilon }\left( x-a \right)^{1-\beta} \leq \frac{I_3}{\epsilon} \leq 
   	\frac{  M_{x +\epsilon , x} }{\left( 1-\beta \right) \epsilon }\left( x-a \right)^{1-\beta} 
   	\]
   	By the absolute continuity of the integral the quotient limit
   	$
   	\frac{I_3}{\epsilon} 
   	$
   	 exists as $\epsilon \rightarrow 0$ for almost all $x$.
   	 This also implies the existence of the other two limits.
   	 Therefore, the following bond holds 
   	\[
   	m^{\star}_{x +\epsilon , x} \frac{\left( x-a \right)^{1-\beta} }{\left( 1-\beta \right)  } \leq \underbrace{ \llim{\epsilon}{0}{\frac{I_3}{\epsilon}} }_{I_3^\prime} \leq   M^{\star}_{x +\epsilon , x} \frac{\left( x-a \right)^{1-\beta} }{\left( 1-\beta \right)  }  
   	\]
   	where 
   	$ M^{\star}_{x +\epsilon, x} = \sup_{[x, x +\epsilon]} f^\prime   $ and $  m^{\star}_{x +\epsilon, x} = \inf_{[x, x +\epsilon]} f^\prime $ wherever these exist.
   	Therefore, as $x$ approaches $a$ 
   	$
   	\llim{x}{a}{I_3^\prime}=0
   	$.
   	
   	Finally, we establish the bounds of the limit
   	\[
   	 \llim{x}{a}{}\frac{  m_{x , a}}{ \left( x-a\right)^\beta } \leq \llim{x}{a}{I_2^\prime} \leq \llim{x}{a}{}\frac{  M_{x , a}}{ \left( x-a\right)^\beta } \epnt
   	\]
   	Therefore, condition \ref{C1} is necessary for the existing of the limit and hence for
   	$\llim{x}{a}{I^\prime}$ .
 
   \end{proof}
 
   Based on this result, we can state a generic continuity result for LFD of fractional order. 
   \begin{theorem}[Continuity of LFD]\label{th:fdiffcont1}
   	 For $ 0<\beta<1$ if $\mathcal{D}_{KG \pm}^{\beta}  f(x) $ is continuous about \textit{x} then $\mathcal{D}_{KG \pm}^{\beta}  f(x) = 0$.
   \end{theorem}
   \begin{proof}
   We will prove the case for $\mathcal{D}_{KG +}^{\beta}  f(x) $.
   Suppose that LFD is continuous in the interval $[a, x]$ and $\mathcal{D}_{KG +}^{\beta}  f(a)= K  \neq 0 $. 
   Then the conditions of  Lemma  \ref{th:KGconverse1} apply, that is $f \in \holder{\beta}$ a.e. in $[a, x]$.
   Therefore, without loss of generality we can assume that $f \in \holder{\beta}$ at $a$. 
   Further, we express the R-L derivative in Euler form setting $z=x-a$ :
   \[
   \mathcal{D}_{z}^{\beta} f  = \frac{\partial}{\partial z} \, \frac{z^{1-\beta}}{\Gamma(1-\beta)} \int\limits^1_0  \frac{f \left(z \, (1-t) +a \right)- f\left(a\right)}{t^\beta} dt
   \]
   By the monotonicity of the power function (e.g.  H\"older growth property):  
   \[
   k_1   \Gamma(1+\beta) \leq \mathcal{D}_{z}^{\beta} f \leq K_1   \Gamma(1+\beta)
   \]
   where $k_1 = \inf_{[a, a+z]} f - f(a)$ and  $K_1 = \sup_{[a, a+z]} f - f(a) $.
   On the other hand,
   we can split the integrand in two expressions for an arbitrary intermediate value $z_0=\lambda z \leq z$.
   This gives
   \begin{flalign*}
   \mathcal{D}_{z}^{\beta} f=\frac{\partial}{\partial z}  &\, \frac{z^{1-\beta}}{\Gamma(1-\beta)} \int\limits^1_0  \frac{f \left(z \, (1-t) +a \right)- f\left( \lambda z \, (1-t) +a \right)}{t^\beta} dt \ + \\
    \frac{\partial}{\partial z} &\, \frac{z^{1-\beta}}{\Gamma(1-\beta)} \int\limits^1_0  \frac{f \left(\lambda z \, (1-t) +a \right)- f\left(a\right)}{t^\beta} dt \epnt
   \end{flalign*}
   Therefore, by the H\"older growth property and monotonicity in $z$
   \[
   \mathcal{D}_{z}^{\beta} f \leq \frac{\partial}{\partial z} z \left(1 - \lambda \right)^\beta \Gamma (1+\beta) K_{1-\lambda} + \frac{\partial}{\partial z} z \lambda^\beta \Gamma (1+\beta) K_{\lambda}  \epnt
   \]
   where $K_\lambda = \sup_{[a, a+\lambda z]} f - f(a)$ and  $K_{1-\lambda} = \sup_{[a+\lambda z, a+z]} f - f(a+\lambda z) $. 
   Therefore, 
     \[
      k_1  \Gamma(1+\beta) \leq  \mathcal{D}_{z}^{\beta} f \leq   \left(
   \left(1 - \lambda \right)^\beta K_{1-\lambda}+  \lambda^\beta K_{\lambda} \right)  \Gamma(1+\beta) \epnt
     \]
     However, by the assumption of continuity 
     $k_1 = K_{\lambda} =K_{1-\lambda} =K $ as $z \rightarrow 0$ and the non-strict inequalities become equalities so that
     \[
       \left(     \left(1 - \lambda \right)^\beta +  \lambda^\beta  - 1 \right)  K = 0 \epnt
     \]
     However, if $\beta <1$ we have contradiction since then $\lambda=1$ or $\lambda=0$ must hold and $\lambda$ seizes to be arbitrary. 
     Therefore, since $\lambda$ is arbitrary $K=0$ must hold.
     The right case can be proven in a similar manner. 
   \end{proof}
   \begin{corollary}[Discontinuous LFD]\label{th:discontd1}
   	Let $\chi_{\beta}:= \{x:  \mathcal{D}_{KG \pm}^{\beta}  f(x) \neq 0 \}$.
   	Then for $ 0<\beta<1$ $\chi_{\beta}$ is totally disconnected. 
   \end{corollary}
    \begin{remark}
   This result is related to Corr. 3 in \cite{Chen2010} however here it is established in a more general way.
   \end{remark}

	\subsection{Equivalent forms of LFD}\label{sec:lfdcalc}

	LFD can be calculated in the following way.	
	Starting from formula \ref{eq:eulerint} for convenience we define the integral average
	\begin{equation}\label{eq:ma}
		M_a (h) :=  \int\limits_0^1 \frac{ f( h u +a ) - f(a)}{(1-u)^{\beta} } du 
	\end{equation}

	Then
	$$
	\Gamma(1- \beta) \mathcal{D}_{KG+}^{\beta} f(a)= \llim{h}{0}{ \underbrace{  h^{1-\beta} \frac{\partial }{\partial h} M_a(h)}_{N_h} }+
	\left(1-\beta \right) \llim{h}{0}{  \frac{ M_a(h) }{ h^\beta } 	}
    $$
    Then we apply L'H\^opital's rule on the second term : 
    $$
    \Gamma(1- \beta) \mathcal{D}_{KG+}^{\beta} f(a) = \llim{h}{0}{N_h} + \frac{1 -\beta}{\beta} \llim{h}{0}{ \underbrace{  h^{1-\beta} \frac{\partial }{\partial h} M_a(h)}_{N_h} } 
    = \frac{1}{\beta} \llim{h}{0}{N_h} 
    $$
    Finally,
    \begin{equation}
     \mathcal{D}_{KG+}^{\beta} f(a) =\frac{1}{\beta \,  \Gamma (1- \beta)}  \llim{h}{0} {  h^{1-\beta} \frac{\partial }{\partial h} \ \int\limits_0^1 \frac{ f( h u +a ) - f(a)}{(1-u)^{\beta} } du }
    \end{equation}
    
    From this equation there are two conclusions that can be drawn
    
    First, for  $f \in L^1 ( a, x)$ by application of the definition of fractional velocity and  L'H\^opital's rule:
    \begin{equation}\
    \mathcal{D}_{KG+}^{\beta} f(a) = \frac{ \fdiffplus{M_a}{0} {\beta}}{ \Gamma (1-\beta)} 
    \end{equation}
    if the last limit exists. 
    Therefore, LFD can be characterized in terms of fractional velocity.
    This can be formalized in the following proposition:
	\begin{proposition}
		Suppose that $f \in \mathcal{I}^{\alpha}_{a, \pm} (L^1) $ for $x \in [ a,  a+ \delta)$ (resp. $ x \in ( a- \delta, a]$ ) for some small $\delta>0$.
	    If $\fdiffpm{M_a}{0} {\beta}$ exists finitely then
	    $$\mathcal{D}_{KG \pm}^{\beta} f(a) =   \frac{\fdiffpm{M_a}{0} {\beta}}{\Gamma (1-\beta)}$$
	     where $M_a (h)$ is given  by formula \ref{eq:ma}.
    \end{proposition}
     From this we see that $f$ may not  be $\beta$-differentiable at $x$. 
     In this perspective LFD is a derived concept - it is the $\beta-$ velocity of the integral average.
      
    Second,
    for BV functions the order of integration and parametric derivation can be exchanged so that
    \begin{equation}\label{eq:KGbv}
		\mathcal{D}_{KG+}^{\beta} f(a) =\frac{1}{\beta \,  \Gamma (1- \beta)}  \llim{h}{0} {  h^{1-\beta}   \ \int\limits_0^1 \frac{ u f^\prime( h u +a ) }{(1-u)^{\beta} } du }
   \end{equation}
    where we demand the existence of $ f^\prime( x )  $ a.e in $(a, x)$, which follows from the Lebesgue differentiation theorem. 
    This statement can be formalized as
    \begin{proposition}
    Suppose that $f \in  BV (a, \delta]$ for some small $\delta>0$.
    Then 
    \[
    \mathcal{D}_{KG+}^{\beta} f(a) =\frac{1}{\beta \,  \Gamma (1- \beta)}  \llim{h}{0} {  h^{1-\beta}   \ \int\limits_0^1 \frac{ u f^\prime( h u +a ) }{(1-u)^{\beta} } du }
    \]
    \end{proposition}
    In the last two formulas we can also set $\Gamma(-\beta) = \beta \, \Gamma (1- \beta) $ by the reflection formula.

    Therefore, in the conventional form for a BV function 
	 \begin{equation}
	 \mathcal{D}_{KG+}^{\beta} f(a) =\frac{1}{\beta \,  \Gamma (1- \beta)}  \llim{x}{a+} {  (x-a)^{1-\beta} \frac{\partial }{\partial x} \ \int\limits_0^1 \frac{ f( (x-a) u +a ) - f(a)}{(1-u)^{\beta} } du }
	 \end{equation}
    
   \section{Discussion} 	
	\label{sec:disc}
    
    Kolwankar-Gangal local fractional derivative was introduced as a tool for study of the scaling of physical systems and systems exhibiting fractal behavior \cite{Kolwankar1998}.  
  	The conditions for applicability of the K-G fractional derivative were not specified in the seminal paper, which leaves space for different interpretations and sometimes confusions. 
  	For example, recently Tarasov claimed that local fractional derivatives of fractional order vanish everywhere 
  	\cite{Tarasov2016}. 
  	In contrast, the results presented here demonstrate that local fractional derivatives vanish only if they are continuous.
  	Moreover, they are non-zero on arbitrary dense sets of measure zero for $\beta$-differentiable functions as shown.
  	
  	Another confusion is the initial claim presented in \cite{Adda2001} that   K-G fractional  derivative is equivalent to what is called here $\beta$-fractional velocity needed to be clarified in \cite{Chen2010} and restricted to the more limited functional space of summable fractional Riemann-Liouville derivatives \cite{Adda2013}.

	Presented results call for a careful inspection of the claims branded under the name of "local fractional calculus" using K-G fractional derivative. 
	Specifically, in the  implied conditions on image function's regularity and arguments of continuity of resulting \textsl{local fractional derivative} must be examined in all cases. 
    For example, in another stream of literature fractional difference quotients are defined on fractal sets, such as the Cantor's set \cite{Yang2015}.
    This is not to be confused with the original approach of Cherebit, Kolwankar and Gangal where the topology is of the real line
    and the set $\chi_{\alpha}$ is totally disconnected.

	\section{Conclusion}
	\label{sec:conclusion}
	
	 As demonstrated here, fractional velocities can be used to characterize the set of change of F-analytic functions.
	Local fractional derivatives and the equivalent fractional velocities have several distinct properties compared to integer-order derivatives. 
	This may induce some wrong expectations to uninitiated reader.
	Some authors can even argue that these concepts are not suitable tools to deal with  non-differentiable functions.
	However, this view pertains only to expectations transfered from the behavior of ordinary derivatives. 
	On the contrary, one-sided local fractional derivatives can be used as a tool to study \textbf{ local non-linear behavior} of functions as demonstrated by the presented examples.
	In applied problems, local fractional derivatives can be also used  to derive fractional Taylor expansions \cite{Liu2015, Prodanov2016a, Prodanov2017}.

	\section*{Acknowledgments}
	The work has been supported in part by a grant from Research Fund - Flanders (FWO), contract number  VS.097.16N.
 
	\appendix
	
	\section{Essential properties of fractional velocity}
	\label{sec:appendix}
	
	In this section we assume that the functions are BVC in the neighborhood of the point of interest.
	Under this assumption we have
	
	\begin{itemize}
		\item Product rule
			\begin{flalign*}
		\fdiffplus{ [f \, g]  }{x}{\beta}   & =   \fdiffplus{ f}{x}{\beta}  g (x) +   \fdiffplus{ g}{x}{\beta} f(x) 	+	[f,g]^{+}_\beta(x)  \\
		\fdiffmin{ [f \, g]  }{x}{\beta}   & =   \fdiffmin{ f}{x}{\beta}  g (x) +   \fdiffmin{ g}{x}{\beta} f(x) 	-	[f,g]^{-}_\beta(x) 
		\end{flalign*}
		\item Quotient rule
		\begin{flalign*}
		\fdiffplus{   [f / g ]   }{x}{\beta}    & =  \frac{  \fdiffplus{ f}{x}{\beta}  g (x) -  \fdiffplus{ g}{x}{\beta} f(x) -[f,g]^{+}_{\beta} }{g^2(x)}	 \\
		\fdiffmin{ [f / g ]  }{x}{\beta}    & =   \frac{\fdiffmin{ f}{x}{\beta}  g (x) -   \fdiffmin{ g}{x}{\beta} f(x) + [f,g]^{-}_{\beta}  }{g^2(x)}	 
		\end{flalign*}
		
	\end{itemize}
	where 
\[ 
	[f,g]_\beta^{  \pm}(x) :=  \llim{\epsilon}{0}{}    \fracvar{f}{x}{\beta/2}{\epsilon \pm} \, \fracvar {g}{x}{\beta/2}{\epsilon \pm} 
	\] 
	
	For compositions of functions
	\begin{itemize}
		\item  $f \in \holder{\beta}$ and $g \in \fclass{C}{1}$
	\begin{flalign*}
	\fdiffplus{  f \circ g  }{x}{\beta}    & =  	\fdiffplus{  f   }{g}{\beta} \left(  g^{\prime  } (x) \right) ^\beta	 \\
	\fdiffmin{  f \circ g  }{x}{\beta}    & =  	\fdiffmin{  f   }{g}{\beta} \left(  g^{\prime  } (x) \right) ^\beta
	\end{flalign*}
		\item  $f \in \fclass{C}{1}$ and  $g \in \holder{\beta}$ 
		\begin{flalign*}
		\fdiffplus{  f \circ g  }{x}{\beta}    & =   f^{\prime } (g) \,	\fdiffplus{ g }{x}{\beta}  	 \\
		\fdiffmin{  f \circ g  }{x}{\beta}    & =  	f^{\prime } (g) \,\fdiffmin{  g   }{x}{\beta}  
		\end{flalign*}
	\end{itemize}

	Basic evaluation formula
	\[
		\fdiffpm{ f}  {x}{\beta}  = \frac{1}{\beta}\llim{\epsilon}{0}{\epsilon^{1-\beta} f^{\prime} (x \pm \epsilon)}
	\]
\bibliographystyle{plain}  
\bibliography{qvar1}

\begin{thebibliography}{10}

\bibitem{Adda2001}
F.~Ben Adda and J.~Cresson.
\newblock About non-differentiable functions.
\newblock {\em J. Math. Anal. Appl.}, 263:721 -- 737, 2001.

\bibitem{Adda2013}
F.~Ben Adda and J.~Cresson.
\newblock Corrigendum to "{A}bout non-differentiable functions" [{J. Math.
  Anal. Appl.} 263 (2001) 721 -- 737].
\newblock {\em J. Math. Anal. Appl.}, 408(1):409 -- 413, 2013.

\bibitem{Caputo1971}
M.~Caputo and F.~Mainardi.
\newblock Linear models of dissipation in anelastic solids.
\newblock {\em Rivista del Nuovo Cimento}, 1:161 -- 198, 1971.

\bibitem{Cesaro1906}
E.~Ces\`aro.
\newblock Fonctions continues sans d\'eriv\'ee.
\newblock {\em Archiv der Math. und Phys.}, 10:57-- 63, 1906.

\bibitem{Chen2010}
Y.~Chen, Y.~Yan, and K.~Zhang.
\newblock On the local fractional derivative.
\newblock {\em J. Math. Anal. Appl.}, pages 17 -- 33, 2010.

\bibitem{Cherbit1991}
G.~Cherbit.
\newblock {\em Fractals, Non-integral dimensions and applications}, chapter
  Local dimension, momentum and trajectories, pages 231-- 238.
\newblock John Wiley \& Sons, Paris, 1991.

\bibitem{Cresson2016}
J.~Cresson and F.~Pierret.
\newblock Multiscale functions, scale dynamics, and applications to partial
  differential equations.
\newblock {\em Journal of Mathematical Physics}, 57(5):053504, 2016.

\bibitem{Darst2009}
R.~Darst, J.. Palagallo, and T.~Price.
\newblock {\em Curious Curves}.
\newblock World Scientific Publishing Company, 2009.

\bibitem{Rham1957}
G.~de~Rham.
\newblock Sur quelques courbes definies par des equations fonctionnelles.
\newblock {\em Univ. e Politec. Torino. Rend. Sem. Mat.}, 16:101 -- 113, 1957.

\bibitem{BoisReymond1875}
P.~du~Bois-Reymond.
\newblock Versuch einer classification der willkürlichen functionen reeller
  argumente nach ihren aenderungen in den kleinsten intervallen.
\newblock {\em J Reine Ang Math}, 79:21--37, 1875.

\bibitem{Faber1909}
G.~Faber.
\newblock \"{U}ber stetige funktionen.
\newblock {\em Math. Ann.}, 66:81 -- 94, 1909.

\bibitem{Gillespie1996}
D.~T. Gillespie.
\newblock {The mathematics of Brownian motion and Johnson noise}.
\newblock {\em American Journal of Physics}, 64(3):225--240, 1996.

\bibitem{Gorenflo2008}
R.~Gorenflo and F.~Mainardi.
\newblock Continuous time random walk, {Mittag-Leffler} waiting time and
  fractional diffusion: Mathematical aspects.
\newblock In {\em Anomalous Transport}, pages 93--127. Wiley-{VCH} Verlag
  {GmbH} {\&} Co. {KGaA}, 2008.

\bibitem{Hutchinson1981}
John Hutchinson.
\newblock Fractals and self similarity.
\newblock {\em Indiana University Mathematics Journal}, 30(5):713, 1981.

\bibitem{Kolwankar1997}
K.~M. Kolwankar and A.D. Gangal.
\newblock Fractional differentiability of nowhere differentiable functions and
  dimensions.
\newblock {\em Chaos}, 6(4):505 -- 513, 1996.

\bibitem{Kolwankar1998}
K.~M. Kolwankar and A.D. Gangal.
\newblock Local fractional {Fokker-Planck} equation.
\newblock {\em Phys. Rev. Lett.}, 80:214--217, Jan 1998.

\bibitem{Kolwankar2001}
K.~M. Kolwankar and J.~L\'evy V\'ehel.
\newblock Measuring functions smoothness with local fractional derivatives.
\newblock {\em Frac. Calc. Appl. Anal.}, 4(3):285 -- 301, 2001.

\bibitem{Liu2015}
Z.~Liu, T.~Wang, and G.~Gao.
\newblock A local fractional {Taylor} expansion and its computation for
  insufficiently smooth functions.
\newblock {\em East Asian Journal on Applied Mathematics}, 5(02):176--191, may
  2015.

\bibitem{Lomnicki1934}
Z.~Lomnicki and S.~Ulam.
\newblock Sur la th\'eorie de la mesure dans les espaces combinatoires et son
  application au calcul des probabilit\'es i. variables ind\'ependantes,.
\newblock {\em Fund. Math.}, 23:237 -- 278, 1934.

\bibitem{Losa1996}
G.~Losa and T.~Nonnenmacher.
\newblock Self-similarity and fractal irregularity in pathologic tissues.
\newblock {\em Modern pathology}, 9(3):174--182, 1996.

\bibitem{Mainardi1997}
F.~Mainardi.
\newblock {\em Fractals and Fractional Calculus in Continuum Mechanics},
  chapter Fractional Calculus: Some Basic Problems in Continuum and Statistical
  Mechanics, pages 291 -- 348.
\newblock Springer, Wien and New York, 1997.

\bibitem{Mandelbrot1982}
B.~Mandelbrot.
\newblock {\em Fractal Geometry of Nature}.
\newblock HENRY HOLT \& CO, 1982.

\bibitem{Mandelbrot1989}
B.~Mandelbrot.
\newblock {\em Les objets fractals: Forme, hasard et dimension}.
\newblock Flammarion, 1989.

\bibitem{Metzler2004}
R.~Metzler and J.~Klafter.
\newblock The restaurant at the end of the random walk: recent developments in
  the description of anomalous transport by fractional dynamics.
\newblock {\em Journal of Physics A: Mathematical and General}, 37(31):R161,
  2004.

\bibitem{Neidinger2016}
Richard~D. Neidinger.
\newblock A fair-bold gambling function is simply singular.
\newblock {\em The American Mathematical Monthly}, 123(1):3, 2016.

\bibitem{Nottale2010}
L.~Nottale.
\newblock Scale relativity and fractal space-time: Theory and applications.
\newblock {\em Foundations of Science}, 15:101--152, 2010.

\bibitem{Odibat2007}
Z.~M. Odibat and N.~T. Shawagfeh.
\newblock Generalized {Taylor's} formula.
\newblock {\em Appl. Math. Comput.}, 186(1):286--293, mar 2007.

\bibitem{Oldham1974}
K.B. Oldham and J.S. Spanier.
\newblock {\em The Fractional Calculus: Theory and Applications of
  Differentiation and Integration to Arbitrary Order}.
\newblock Academic Press, New York, 1974.

\bibitem{Prodanov2015}
D.~Prodanov.
\newblock Fractional variation of {H}\"olderian functions.
\newblock {\em Fract. Calc. Appl. Anal.}, 18(3):580 -- 602, 2015.

\bibitem{Prodanov2016b}
D.~Prodanov.
\newblock Characterization of strongly non-linear and singular functions by
  scale space analysis.
\newblock {\em Chaos, Solitons \& Fractals}, 93:14--19, Dec 2016.

\bibitem{Prodanov2016a}
D.~Prodanov.
\newblock Regularization of derivatives on non-differentiable points.
\newblock {\em Journal of Physics: Conference Series}, 701(1):012031, 2016.

\bibitem{Prodanov2017}
D.~Prodanov.
\newblock Conditions for continuity of fractional velocity and existence of
  fractional {Taylor} expansions.
\newblock {\em Chaos, Solitons {\&} Fractals}, 102:236--244, sep 2017.

\bibitem{Salem1943}
R.~Salem.
\newblock On some singular monotonic functions which are strictly increasing.
\newblock {\em Trans. Am Math. Soc.}, 53(3):427 -- 439, 1943.

\bibitem{Samko1993}
S.~Samko, A.~Kilbas, and O.~Marichev, editors.
\newblock {\em Fractional Integrals and Derivatives: Theory and Applications}.
\newblock Gordon and Breach, Yverdon, Switzerland,, 1993.

\bibitem{Schroeder1991}
M.~Schroeder.
\newblock {\em Fractals, Chaos, Power Laws: Minutes from an Infinite Paradise}.
\newblock Dover Publications, 1991.

\bibitem{Tarasov2016}
V.E. Tarasov.
\newblock Local fractional derivatives of differentiable functions are
  integer-order derivatives or zero.
\newblock {\em International Journal of Applied and Computational Mathematics},
  2(2):195--201, 2016.

\bibitem{Wheatcraft2008}
S.~W. Wheatcraft and M.~M. Meerschaert.
\newblock Fractional conservation of mass.
\newblock {\em Adv. Water Resour.}, 31(10):1377--1381, oct 2008.

\bibitem{Yang2015}
X.~J. Yang, D.~Baleanu, and H.~M. Srivastava.
\newblock {\em Local Fractional Integral Transforms and Their Applications}.
\newblock Academic Press, 2015.

\end{thebibliography}

\end{document}